\RequirePackage{fix-cm}
\documentclass{siamart190516}
\newsiamremark{remark}{Remark}
\usepackage{algorithmic} 
\Crefname{ALC@unique}{Line}{Lines} 

\usepackage{amsmath}
\usepackage{amssymb}
\usepackage{mathtools}
\usepackage{cases}
\usepackage{fancyvrb}
\VerbatimFootnotes
\usepackage{pgfplots, pgfplotstable, booktabs, colortbl, array}
\usepackage{bibentry}
\nobibliography*
\usepackage{enumitem}
\usepackage{tabularx}
\newcolumntype{Y}{>{\centering\arraybackslash}X}

\usepackage[T1]{fontenc}

\newcommand{\textoverline}[1]{$\overline{\mbox{#1}}$}

\usepackage{pgfplots, pgfplotstable, booktabs, colortbl, array}
\usepackage{tikz}
\pgfplotsset{compat=newest}
\usepackage{pgffor}
\usepackage{xcolor}

\usepackage{multirow}
\newcolumntype{C}[1]{>{\centering\arraybackslash}p{#1}}

\DeclareMathOperator*{\argmin}{arg\,min}

\renewcommand{\Re}{\operatorname{Re}}

\DeclareMathOperator*{\sn}{sn}
\DeclareMathOperator*{\cn}{cn}
\DeclareMathOperator*{\dn}{dn}
\DeclareMathOperator*{\sign}{sign}
\DeclareMathOperator{\arctanh}{arctanh}

\makeatletter
\newcommand{\LINEIF}[3][default]{%
  \ALC@it\algorithmicif\ #2\ \algorithmicthen%
  \ #3\  \algorithmicendif%
}
\makeatother

\makeatletter
\newcommand{\LINEIFELSE}[4][default]{%
  \ALC@it\algorithmicif\ #2\ \algorithmicthen%
  \ #3\ \algorithmicelse\ #4  \algorithmicendif%
}
\makeatother

\usepackage{tikz}
\usetikzlibrary{positioning,arrows}

\usepgfplotslibrary{fillbetween}

\begin{document}

\title{Iterations for the Unitary Sign Decomposition and the Unitary Eigendecomposition}

\author{Evan S. Gawlik\thanks{Department of Mathematics, University of Hawai`i at M\textoverline{a}noa (\email{egawlik@hawaii.edu})}\funding{The author was supported by NSF grants DMS-1703719 and DMS-2012427.}}
    
\date{}

\headers{Iterations for the Unitary Sign Decomposition}{Evan S. Gawlik}

\maketitle

\begin{abstract}
We construct fast, structure-preserving iterations for computing the sign decomposition of a unitary matrix $A$ with no eigenvalues equal to $\pm i$.  This decomposition factorizes $A$ as the product of an involutory matrix $S = \operatorname{sign}(A) = A(A^2)^{-1/2}$ times a matrix $N = (A^2)^{1/2}$ with spectrum contained in the open right half of the complex plane.  Our iterations rely on a recently discovered formula for the best (in the minimax sense) unimodular rational approximant of the scalar function $\operatorname{sign}(z) = z/\sqrt{z^2}$ on subsets of the unit circle.  When $A$ has eigenvalues near $\pm i$, the iterations converge significantly faster than Pad\'e iterations.  Numerical evidence indicates that the iterations are backward stable, with backward errors often smaller than those obtained with direct methods.  This contrasts with other iterations like the scaled Newton iteration, which suffers from numerical instabilities if $A$ has eigenvalues near $\pm i$.  As an application, we use our iterations to construct a stable spectral divide-and-conquer algorithm for the unitary eigendecomposition.
\end{abstract}



\section{Introduction} \label{sec:intro}

Every matrix $A \in \mathbb{C}^{n \times n}$ with no purely imaginary eigenvalues can be written uniquely as a product
\[
A = SN,
\]
where $S \in \mathbb{C}^{n \times n}$ is involutory ($S^2=I$), $N \in \mathbb{C}^{n \times n}$ has spectrum in the open right half of the complex plane, and $S$ commutes with $N$. This is the celebrated matrix sign decomposition~\cite{higham1994matrix}, whose applications are widespread~\cite{denman1976matrix,kenney1995matrix}. In terms of the principal square root $(\cdot)^{1/2}$, we have $S=A(A^2)^{-1/2} =: \mathrm{sign}(A)$ and $N=(A^2)^{1/2}$.

When $A$ is unitary, so too are $S$ and $N$.  It follows that $S=S^{-1}=S^*$, so we may write, for any unitary $A$ with $\Lambda(A) \cap i\mathbb{R} = \emptyset$,
\begin{equation} \label{unitarysign}
A = SN, \quad S^2=I, \; S=S^*, \; N^2=A^2, \; N^*N=I, \, \Lambda(N) \subset \mathbb{C}_+,
\end{equation}
where $\Lambda(N)$ denotes the spectrum of $N$ and $\mathbb{C}_+ = \{z \in \mathbb{C} \mid \Re(z)>0\}$.  We refer to this decomposition as the \emph{unitary sign decomposition}.

We say that an algorithm for computing the decomposition~(\ref{unitarysign}) is backward stable if it computes matrices $\widehat{S}$ and $\widehat{N}$ with the property that the quantities
\begin{equation} \label{backwarderrors}
\|A-\widehat{S}\widehat{N}\|, \, \|\widehat{S}^2-I\|, \, \|\widehat{S}-\widehat{S}^*\|, \, \|\widehat{N}^*\widehat{N}-I\|, \, \|\widehat{N}^2-A^2\|, \, \max\{0,-\min_{\lambda \in \Lambda(\widehat{N})} \Re\lambda\}
\end{equation}
are each a small multiple of the unit roundoff $u$ ($=2^{-53}$ in double-precision arithmetic).\footnote{Note that this property implies $\|\widehat{N}\widehat{S}-\widehat{S}\widehat{N}\|$ is small as well; see Lemma~\ref{lemma:SNcommute}.}   Here, $\|\cdot\|$ denotes the 2-norm.

The goal of this paper is to design backward stable iterations for computing the decomposition~(\ref{unitarysign}).  To illustrate why this is challenging, let us point out the pitfalls of naive approaches.  A widely used iteration for computing the sign of a general matrix $A \in \mathbb{C}^{n \times n}$ is the Newton iteration~\cite{roberts1980linear}~\cite[Section 5.3]{higham2008functions}
\begin{equation} \label{newtonit}
X_{k+1} = \frac{1}{2}(X_k+X_k^{-1}), \quad X_0 = A.
\end{equation}
If $A$ is unitary, then the first iteration is simply
\begin{equation} \label{firstnewtonit}
X_1 = \frac{1}{2}(A+A^*).
\end{equation}
In floating point arithmetic, this calculation is susceptible to catastrophic cancellation if $A$ has eigenvalues near $\pm i$.  Indeed, if we carry out~(\ref{firstnewtonit}) followed by~(\ref{newtonit}) for $k=1,2,\dots$ on the $100 \times 100$ unitary matrix \verb$A = gallery('orthog',100,3)$ from the MATLAB matrix gallery, then the iteration diverges.  Scaling the iterates with standard scaling heuristics~\cite{kenney1992scaling} leads to convergence, but the computed sign of $A$ satisfies $\|\widehat{S}-\widehat{S}^*\| > 0.1$ in typical experiments.  
This happens because $A$ has several eigenvalues lying near $\pm i$.

The above algorithm can be reinterpreted in a different way:  It is computing the unitary factor in the polar decomposition of $(A+A^*)/2$.  Indeed, the Newton iteration $X_{k+1} = \frac{1}{2}(X_k+X_k^{-*})$ for the polar decomposition~\cite[Section 8.3]{higham2008functions} coincides with~(\ref{newtonit}) on Hermitian matrices.  This suggests another   family of potential algorithms:  compute the polar decomposition of $(A+A^*)/2$ via iterative methods or other means.  However, numerical experiments confirm that such algorithms are similarly inaccurate on matrices with eigenvalues near $\pm i$.  This unstable behavior is also shared by the superdiagonal Pad\'e iterations for the matrix sign function~\cite{kenney1991rational}, all of which map eigenvalues $\lambda \approx \pm i$ of $A$ to a small real number (or the inverse thereof) in the first iteration.

One way to overcome these difficulties is to adopt structure-preserving iterations.  Here, we say that an iteration $X_{k+1} = g_k(X_k)$ for the unitary sign decomposition is structure-preserving if the iterates $X_k$ are unitary for every $k$.  
Examples include the diagonal family of Pad\'e iterations~\cite{higham2004computing}, whose lowest-order member is the iteration
\begin{equation} \label{padelow}
X_{k+1} = X_k(3I+X_k^2)(I+3X_k^2)^{-1}, \quad X_0 = A.
\end{equation}
By keeping $X_k$ unitary, a structure-preserving iteration ensures that the eigenvalues of $X_k$ remain on the unit circle, ostensibly skirting the dangers of catastrophic cancellation.  We observe numerically that, if implemented in a clever way (described in Section~\ref{sec:algorithm}), the diagonal Pad\'e iterations are backward stable.  However, they can take excessively long to converge on matrices with eigenvalues near $\pm i$.  For example, when \verb$A = gallery('orthog',100,3)$, the iteration~(\ref{padelow}) takes 34 iterations to converge.

We construct in this paper a family of structure-preserving iterations for the unitary sign decomposition that converge more rapidly---sometimes dramatically more so---than the diagonal Pad\'e iterations.   Numerical evidence indicates that these iterations are backward stable, with backward errors often smaller than those obtained with direct methods.

The key ingredient that we use to construct our iterations is a recently discovered formula for the best (in the minimax sense) unimodular rational approximant of the scalar function $\sign(z) = z/\sqrt{z^2}$ on subsets of the unit circle~\cite{gawlik2020zolotarev}.  Remarkably, it can be shown that composing two such approximants yields a best approximant of higher degree~\cite{gawlik2020zolotarev}, laying the foundations for an iteration.  When applied to matrices, the iteration produces a sequence of unitary matrices $X_0=A$, $X_1$, $X_2$, $\dots$ that converges rapidly to $S=\sign(A)$, often significantly faster than the corresponding diagonal Pad\'e iteration.  When \verb$A = gallery('orthog',100,3)$, for example, the lowest-order iteration converges within 6 iterations, which is about 6 times faster than the corresponding diagonal Pad\'e iteration~(\ref{padelow}).

\paragraph{Prior work}
Matrix iterations constructed from rational minimax approximants have attracted growing interest in recent years.  Early examples include the optimal scaling heuristic proposed by Byers and Xu~\cite{byers2008new} for the Newton iteration for the polar decomposition, as well as an analogous scaling heuristic for the matrix square root proposed by Wachspress~\cite{wachspress1962} and Beckermann~\cite{beckermann2013optimally}.  Nakatsukasa, Bai, and Gygi~\cite{nakatsukasa2010optimizing} designed an optimal scaling heuristic for the Halley iteration for the polar decompsition, and their strategy was generalized to higher order by Nakatsukasa and Freund~\cite{nakatsukasa2016computing}.    The latter work elucidated the link between these scaling heuristics and the seminal work of Zolotarev~\cite{zolotarev1877applications} on rational minimax approximation.  The iterations derived in~\cite{nakatsukasa2016computing} have a variety of applications, including algorithms for the symmetric eigendecomposition, singular value decomposition, polar decomposition, and CS decomposition~\cite{nakatsukasa2016computing,gawlik2018backward}.  

All of the aforementioned algorithms rely crucially on the following fact: if two rational minimax approximants of the scalar function $\sign(x)$ on suitable real intervals are composed with one another, then their composition is a best approximant of higher degree~\cite{nakatsukasa2016computing}.  
A related composition law for rational minimax approximants of $\sqrt{z}$ has been used to construct iterations for the matrix square root~\cite{gawlik2018zolotarev}.  These iterations were generalized to the matrix $p$th root in~\cite{gawlik2020rational} and used to derive approximation theoretic results in~\cite{gawlik2019approximating}.
An even more recent advancement---a composition law for rational minimax approximants of $\sign(z)$ on subsets of the unit circle~\cite{gawlik2020zolotarev}---is what inspired the present paper.

\paragraph{Connections to other iterations}
The iterations we derive in this paper are intimately connected to several existing iterations for the matrix sign function and the polar decomposition.  When applied to a unitary matrix $A$, our iterations produce a sequence of unitary matrices whose Hermitian part coincides with the sequence of matrices generated by Nakatsuka and Freund's iterations~\cite{nakatsukasa2016computing} for the polar decomposition of $(A+A^*)/2$.  A special case of this result is a connection between our lowest-order iteration for $\sign(A)$ and the optimally scaled Halley iteration for the polar decomposition of $(A+A^*)/2$~\cite{nakatsukasa2010optimizing}.  It is important to note that these equivalences hold only in exact arithmetic.  In floating-point arithmetic, our iterations behave very differently from the aforementioned algorithms.

There is also a link between our iterations and the diagonal Pad\'e iterations.  Roughly speaking, our iterations are designed using rational minimax approximants of $\sign(z)$ on two circular arcs containing $\pm 1$.  If these arcs are each shrunk to a point, then the diagonal Pad\'e iterations are recovered.  This helps to explain the slow convergence of the diagonal Pad\'e iterations on unitary matrices with eigenvalues near $\pm i$:  The iterations need to approximate $\sign(z)$ near $z = \pm i$, but they use rational functions that are designed to approximate $\sign(z)$ near $z = \pm 1$.

\paragraph{Unitary eigendecomposition}
Our emphasis on handling eigenvalues near $\pm i$ is not merely pedantic.  It is precisely the sort of situation that one often encounters if the unitary sign decomposition is used as part of a spectral divide-and-conquer algorithm for the unitary eigendecomposition.  

Indeed, consider a unitary matrix $A \in \mathbb{C}^{m \times m}$ with eigendecomposition $A=V\Lambda V^*$.  The matrix $(I+\sign(A))/2$ is a spectral projector onto the invariant subspace $\mathcal{V}_+$ of $A$ associated with eigenvalues having positive real part.
A spectral divide-and-conquer algorithm uses this projector to find orthonormal bases $U_1 \in \mathbb{C}^{m \times m_1}$, $U_2 \in \mathbb{C}^{m \times m_2}$, $m_1+m_2=m$, for $\mathcal{V}_+$ and its orthogonal complement.
Then $\begin{pmatrix} U_1 & U_2 \end{pmatrix}^* A \begin{pmatrix} U_1 & U_2 \end{pmatrix}$ is block diagonal, so recursion can be used to determine $V$ and $\Lambda$.  At each step, scalar multiplication by complex numbers with unit modulus can be used to rotate the spectrum so that it is distributed approximately evenly between the left and right half-planes.  If $A$ has a cluster of nearby eigenvalues, then it is reasonable to expect this process to  center the cluster near $\pm i$ at some step.  This is precisely what we observe in practice, and the ability to compute the unitary sign decomposition quickly and accurately in the presence of eigenvalues near $\pm i$ becomes paramount.

\paragraph{Organization} 
This paper is organized as follows.  We begin in Section~\ref{sec:scalarsign} by studying rational minimax approximants of $\sign(z)$ on the unit circle.  This material is largely drawn from~\cite{gawlik2020zolotarev}, but we add some additional results and insights to relate these approximants to Pad\'e approximants.  Next, we use these approximants to construct matrix iterations for the unitary sign decomposition in Section~\ref{sec:algorithm}.  We illustrate their utility by constructing a spectral divide-and-conquer algorithm for the unitary eigendecomposition in Section~\ref{sec:eig}.  We conclude with numerical examples in Section~\ref{sec:numerical}.

\section{Rational Approximation of the Sign Function on the Unit Circle} \label{sec:scalarsign}

In this section, we study rational approximants of the scalar function $\sign(z) = z/\sqrt{z^2}$ on the set
\[
\mathbb{S}_\Theta = \{z \in \mathbb{C} \mid |z|=1, \, \arg z \notin (\Theta,\pi-\Theta) \cup (-\pi+\Theta,-\Theta) \},
\]
where $\Theta \in (0,\pi/2)$.  Since our ultimate interest is in constructing structure-preserving iterations for the unitary sign decomposition, we focus on rational functions $r$ satisfying $|r(z)|=1$ for $|z|=1$.  We call such rational functions unimodular.  Unimodular rational functions have the property that $r(A)$ is unitary for any unitary matrix $A$.

The problem of determining the best (in the minimax sense) unimodular rational approximant of $\sign(z)$ on $\mathbb{S}_\Theta$ has recently been solved in~\cite{gawlik2020zolotarev}.  To describe the solution, let us introduce some notation.  We use $\mathrm{sn}(\cdot,\ell)$, $\mathrm{cn}(\cdot,\ell)$, and $\mathrm{dn}(\cdot,\ell)$ to denote Jacobi's elliptic functions with modulus $\ell$, and we use $\ell' = \sqrt{1-\ell^2}$ to denote the modulus complementary to $\ell$.  We denote the complete elliptic integral of the first kind by $K(\ell) = \int_0^{\pi/2} (1-\ell^2 \sin^2\theta)^{-1/2} \, d\theta$.  We say that a rational function $r(z)=p(z)/q(z)$ has type $(m,n)$ if $p$ and $q$ are polynomials of degree at most $m$ and $n$, respectively.

\begin{theorem}
Let $\Theta \in (0,\pi/2)$ and $n \in \mathbb{N}_0$.  Among all rational functions $r$ of type $(2n+1,2n+1)$ that satisfy $|r(z)|=1$ for $|z|=1$, the ones which minimize
\[
\max_{z \in \mathbb{S}_\Theta} \left| \arg\left(\frac{r(z)}{\sign(z)}\right) \right|
\]
are
\[
r(z) = r_{2n+1}(z;\Theta) =  z \prod_{j=1}^n \frac{z^2+a_j}{1+a_j z^2}
\]
and its reciprocal, where
\[
a_j = a_j(\Theta) = \left( \frac{\ell \sn(v_j,\ell') + \dn(v_j,\ell')}{\cn(v_j,\ell')} \right)^{2(-1)^{j+n}},
\]
$v_j = \frac{2j-1}{2n+1}K(\ell')$, $\ell = \cos\Theta$, and $\ell' = \sqrt{1-\ell^2} = \sin\Theta$.
\end{theorem}
\begin{proof}
See~\cite[Theorem 2.1 and Remark 2.2]{gawlik2020zolotarev}.
\end{proof}
\begin{remark}
For simplicity, we have chosen to focus only on best unimodular rational approximants of $\sign(z)$ on $\mathbb{S}_\Theta$ of type $(2n+1,2n+1)$ in this paper.  
Best approximants of type $(2n,2n)$ can also be written down; see~\cite{gawlik2020zolotarev} for details.
\end{remark}

The rational function $r_{2n+1}(z;\Theta)$ has the following remarkable behavior under composition.

\begin{theorem} \label{thm:composition}
Let $\Theta \in (0,\pi/2)$, $m,n \in \mathbb{N}_0$, and $\widetilde{\Theta} = \left| \arg(r_{2n+1}(e^{i\Theta}; \Theta)) \right|$.  Then
\[
r_{2m+1}(r_{2n+1}(z;\Theta); \widetilde{\Theta}) = r_{(2m+1)(2n+1)}(z;\Theta).
\]
\end{theorem}
\begin{proof}
See~\cite[Theorem 3.3 and Remark 3.6]{gawlik2020zolotarev}.
\end{proof}

We also have the following error estimate.

\begin{theorem} \label{thm:error}
Let $\Theta \in (0,\pi/2)$ and $n \in \mathbb{N}_0$.  We have
\[
\max_{z \in \mathbb{S}_\Theta} \left| \arg\left(\frac{r_{2n+1}(z;\Theta)}{\sign(z)}\right) \right| \le 4 \rho^{-(2n+1)},
\]
where
\begin{equation} \label{rho}
\rho = \rho(\Theta) = \exp\left( \frac{ \pi K(\cos\Theta) }{ 2K(\sin\Theta) } \right).
\end{equation}
\end{theorem}
\begin{proof}
See~\cite[Theorem 3.2]{gawlik2020zolotarev}, and note that their definition of $\rho$ differs from ours by a factor of 2 in the exponent.
\end{proof}

\begin{remark} \label{remark:theta0}
Theorems~\ref{thm:composition} and~\ref{thm:error} continue to hold when $\Theta=0$ if we adopt the convention that $\rho(0) = \infty$, $\mathbb{S}_0 = \{-1,1\}$, and $r_{2n+1}(z;0) = z \prod_{j=1}^n \frac{z^2+a_j(0)}{1+a_j(0) z^2}$.  We elaborate on this fact below.
\end{remark}

\subsection{Connections with Other Rational Approximants} \label{sec:connections}

The rational function $r_{2n+1}(z;\Theta)$ is closely connected to several other well-known rational approximants of $\sign(z)$.  

\begin{proposition} \label{prop:pade}
As $\Theta \rightarrow 0$, $r_{2n+1}(z;\Theta)$ converges coefficientwise to $z p_n(z^2)$, where $p_n(z)$ is the type-$(n,n)$ Pad\'e approximant of $z^{-1/2}$ at $z=1$.
\end{proposition}
\begin{proof}
This is a consequence of~\cite[Proposition 3.9]{gawlik2020zolotarev}, where it is shown that $\sqrt{z}/r_{2n+1}(\sqrt{z};\Theta)$ converges coefficientwise to $1/p_n(z)$ as $\Theta \rightarrow 0$.
\end{proof}

In the notation of Remark~\ref{remark:theta0}, the above proposition states that
\[
r_{2n+1}(z;0) = zp_n(z^2).
\]
This rational function has been studied extensively in~\cite{kenney1991rational,kenney1994hyperbolic,gomilko2012pade}~\cite[Theorem 5.9]{higham2008functions}.  It satisfies~\cite[Theorem 5.9]{higham2008functions}
\begin{equation} \label{tanh}
z p_n(z^2) = \tanh((2n+1)\arctanh z) %= -i \tan((2n+1) \arctan(iz)).
\end{equation}
It also has the following properties. Both $p_n(z)$ and $zp_n(z^2)$ are unimodular~\cite{higham2004computing}; that is, for any $n \in \mathbb{N}_0$,
\[
|zp_n(z^2)| = |p_n(z)|=1, \text{ if } |z|=1.
\]
Under composition, we have~\cite[Theorem 5.9)(c)]{higham2008functions}
\begin{equation} \label{composition0}
r_{2m+1}(r_{2n+1}(z;0);0) = r_{(2m+1)(2n+1)}(z;0)
\end{equation}
for any $m,n \in \mathbb{N}_0$.  Finally, $r_{2n+1}(1;0)=-r_{2n+1}(-1;0)=1$ for all $n \in \mathbb{N}_0$.  These last two facts justify Remark~\ref{remark:theta0}.

The rational functions $r_{2n+1}(z;0)$, $n \in \mathbb{N}_0$, have been used in~\cite{kenney1991rational} to construct iterations for computing the matrix sign function.  The iterations constitute the diagonal family of Pad\'e iterations.   The first few diagonal Pad\'e approximants of $z^{-1/2}$ at $z=1$ are
\[
p_0(z)=1, \; p_1(z) = \frac{3+z}{1+3z}, \; p_2(z) = \frac{5+10z+z^2}{1+10z+5z^2}, \; p_3(z) = \frac{7+35z+21z^2+z^3}{1+21z+35z^2+7z^3}.
\]
More generally, Pad\'e iterations can be constructed from rational functions of the form $zp_{m,n}(z^2)$, where $p_{m,n}(z)$ is the type-$(m,n)$ Pad\'e approximant of $z^{-1/2}$ at $z=1$.   However, when $m\neq n$, the Pad\'e iterations are not structure-preserving, as $|p_{m,n}(z)| \not\equiv 1$ for $|z|=1$ and $m \neq n$.

We now turn our attention back to the rational function $r_{2n+1}(z,\Theta)$ with positive $\Theta$.  Interestingly, this function is intimately connected to the solution of another rational approximation problem: approximating $\sign(x)$ on the union of real intervals $[-1,-\ell] \cup [\ell,1]$.

\begin{theorem} \label{thm:realpart}
Let $\Theta \in [0,\pi/2)$ and $n \in \mathbb{N}_0$.  For $z \in \mathbb{C}$ with $|z|=1$, we have
\begin{equation} \label{realpart}
\Re r_{2n+1}(z;\Theta) = \widehat{R}_{2n+1}(\Re z; \cos\Theta),
\end{equation}
where 
\[
\widehat{R}_m(x;\ell) =
\begin{cases} 
\frac{R_m(x;\ell)}{\max_{y \in [\ell,1]} R_m(y;\ell)} &\mbox{ if } \ell \in (0,1), \\
x p_n(x^2), &\mbox{ if } \ell=1,
\end{cases}
\]
and
\[
R_m(\cdot;\ell) = \argmin_{R \in \mathcal{R}_{m,m}} \max_{x \in [-1,-\ell] \cup [\ell,1]} |R(x)-\sign(x)|.
\]
\end{theorem}
\begin{proof}
This identity is proven for $\Theta \in (0,\pi/2)$ in~\cite[Theorem 2.4]{gawlik2020zolotarev}.  To see that it also holds when $\Theta = 0$, we must show that if $|z|=1$ and $x = \Re z = \frac{1}{2}(z+1/z)$, then
\[
\frac{1}{2} \left( \tanh((2n+1)\arctanh z) + \frac{1}{\tanh((2n+1)\arctanh z)} \right) = \tanh((2n+1)\arctanh x).
\]
Since $\frac{1+x}{1-x} = -\left(\frac{1+z}{1-z}\right)^2$, we have $\arctanh x = \frac{1}{2}\log\left(\frac{1+x}{1-x}\right) = \log\left(i \frac{1+z}{1-z}\right)$.  Thus,
\begin{equation} \label{tanhleft}
\tanh((2n+1)\arctanh x) = \tanh\left( (2n+1)\log\left(i \frac{1+z}{1-z}\right) \right).
\end{equation}
On the other hand, the identity $\tanh(2y) = \frac{2\tanh y}{1+\tanh^2 y}$ shows that
\begin{align}
\frac{1}{2} &\left( \tanh((2n+1)\arctanh z) + \frac{1}{\tanh((2n+1)\arctanh z)}) \right) \nonumber \\
&= \coth((4n+2)\arctanh z) \nonumber \\
&= \coth\left( (2n+1)\log\left(\frac{1+z}{1-z}\right)\right). \label{tanhright}
\end{align}
Since $(2n+1)\log\left(\frac{1+z}{1-z}\right)$ differs from $(2n+1)\log\left(i \frac{1+z}{1-z}\right)$ by an odd multiple of $\frac{\pi i}{2}$, it follows that~(\ref{tanhleft}) and~(\ref{tanhright}) are equal.
\end{proof}

Written another way, the lemma above states that
\begin{equation} \label{rplusrinv}
\frac{1}{2} \left( r_{2n+1}(z;\Theta) + \frac{1}{r_{2n+1}(z;\Theta)} \right) = \widehat{R}_{2n+1}\left( \frac{z+1/z}{2}; \cos\Theta \right)
\end{equation}
for all $z$ with $|z|=1$.  In particular,
\[
\frac{1}{2}\left( zp_n(z^2) + \frac{1}{z p_n(z^2)} \right) = \left( \frac{z+1/z}{2}\right) p_n\left( \left(\frac{z+1/z}{2}\right)^2 \right).
\]
Since these equalities hold on the unit circle, they hold on all of $\mathbb{C}$.

By combining~(\ref{composition0}),~(\ref{realpart}), and Theorem~\ref{thm:composition}, one sees that the function $\widehat{R}_{2n+1}(x;\ell)$ satisfies
\begin{equation} \label{Rcomp}
\widehat{R}_{2m+1}(\widehat{R}_{2n+1}(x,\ell),\widetilde{\ell}) = \widehat{R}_{(2m+1)(2n+1)}(x,\ell), \quad \text{ if } \widetilde{\ell} = \widehat{R}_{2n+1}(\ell,\ell)
\end{equation}
for all $m,n \in \mathbb{N}_0$ and all $\ell \in [0,1)$.
This equality was derived in~\cite{nakatsukasa2016computing} for $\ell \in (0,1)$ by counting extrema of $\widehat{R}_{2m+1}(\widehat{R}_{2n+1}(x,\ell),\widetilde{\ell})-\sign(x)$.
It can be leveraged to construct iterations for the matrix sign function, and such iterations are particularly well-suited for computing the sign of a Hermitian matrix $B$ (which coincides with the unitary factor in the polar decomposition of $B$); see~(\ref{zolopd1}-\ref{zolopd2}) below.

\section{Algorithm} \label{sec:algorithm}

\subsection{Matrix Iteration}

Theorem~\ref{thm:composition} suggests the following iteration for computing the sign of a unitary matrix $A$ with spectrum contained in $\mathbb{S}_\Theta$, $\Theta \in [0,\pi/2)$:
\begin{align}
X_{k+1} &= r_{2n+1}(X_k; \Theta_k), & X_0 &= A,\label{zolo1} \\
\Theta_{k+1} &= |\arg r_{2n+1}(e^{i\Theta_k};\Theta_k)|,  & \Theta_0 &= \Theta. \label{zolo2}
\end{align}

Below we summarize the properties of the iteration~(\ref{zolo1}-\ref{zolo2}).
\begin{proposition}
The iteration~(\ref{zolo1}-\ref{zolo2}) is structure-preserving.  That is, if $A$ is unitary, then $X_k$ is unitary for every $k \ge 0$.
\end{proposition}
\begin{proof}
Since $|r_{2n+1}(z;\Theta_k)|=1$ for every scalar $z$ with unit modulus, $r_{2n+1}(X; \Theta_k)$ is unitary for every unitary matrix $X$.
\end{proof}

\begin{theorem}
Let $A$ be a unitary matrix with spectrum contained in $\mathbb{S}_\Theta$ for some $\Theta \in (0,\pi/2)$.
For any $n \in \mathbb{N}$, the iteration~(\ref{zolo1}-\ref{zolo2}) converges to $\sign(A)$ with order of convergence $2n+1$.  In fact,
\begin{equation} \label{errorbound}
\|\log(X_k \sign(A)^{-1})\| \le 4 \rho^{-(2n+1)^k},
\end{equation}
for every $k \ge 0$, where $\rho$ is given by~(\ref{rho}).
\end{theorem}
\begin{proof}
By Theorem~\ref{thm:composition}, we have
\[
X_k = r_{(2n+1)^k}(A; \Theta)
\]
for every $k \ge 0$.  Thus, every eigenvalue of $X_k \sign(A)^{-1}$ is of the form $r_{(2n+1)^k}(\lambda;\Theta) / \sign(\lambda)$ for some eigenvalue $\lambda$ of $A$.  
By Theorem~\ref{thm:error},
\begin{align*}
\|\log(X_k \sign(A)^{-1})\|
&= \max_{\lambda \in \Lambda(X_k \sign(A)^{-1})} |\arg\lambda| \\
&= \max_{\lambda \in \Lambda(A)} \left| \arg\left( \frac{r_{(2n+1)^k}(\lambda;\Theta)}{\sign(\lambda)} \right)\right| \\
&\le \max_{z \in \mathbb{S}_\Theta} \left| \arg\left(\frac{r_{(2n+1)^k}(z;\Theta)}{\sign(z)}\right) \right| \\
&\le 4 \rho^{-(2n+1)^k}.
\end{align*}
\end{proof}

\subsection{Connections with Other Iterations}

There is an intimate connection between the iteration~(\ref{zolo1}-\ref{zolo2}) and several existing iterations for the matrix sign function.  First, Proposition~\ref{prop:pade} implies that~(\ref{zolo1}-\ref{zolo2}) reduces to the diagonal Pad\'e iteration when we set $\Theta=0$:
\begin{equation} \label{pade}
X_{k+1} = X_k p_n(X_k^2), \quad X_0 = A.
\end{equation}

Second, there is a link between the iteration~(\ref{zolo1}-\ref{zolo2}) and the iteration 
\begin{align}
Y_{k+1} &= \widehat{R}_{2n+1}(Y_k; \ell_k), & Y_0 &= B,\label{zolopd1} \\
\ell_{k+1} &= \widehat{R}_{2n+1}(\ell_k;\ell_k),  & \ell_0 &= \ell, \label{zolopd2}
\end{align}
which was introduced in~\cite{nakatsukasa2016computing} to compute the sign of a Hermitian matrix $B$ with spectrum contained in $[-1,-\ell] \cup [\ell,1]$.  Note that~(\ref{zolopd1}-\ref{zolopd2}) reduces to
\begin{align} 
Y_{k+1} &= Y_k p_n(Y_k^2), &\quad Y_0 &= B \label{padepd}
\end{align}
when we set $\ell=1$ and ignore the spectrum of $B$.  This is the same iteration as~(\ref{pade}), but with a starting matrix labelled $B$ rather than $A$.

\begin{proposition}
Let $A$ be a unitary matrix with no eigenvalues equal to $\pm i$.  Let $n \in \mathbb{N}$ and $\Theta \in [0,\pi/2)$.
If $B = (A+A^*)/2$ and $\ell = \cos\Theta$, then the iterations~(\ref{zolo1}-\ref{zolo2}) and~(\ref{zolopd1}-\ref{zolopd2}) generate sequences satisfying
\[
Y_k = \frac{1}{2}(X_k+X_k^*), \text{ and } \ell_k = \cos\Theta_k
\]
for every $k \ge 0$.
\end{proposition}
\begin{proof}
It follows from Theorem~\ref{thm:composition} that in the iteration~(\ref{zolo1}-\ref{zolo2}), we have
\[
X_k = r_{(2n+1)^k}(A;\Theta), \quad \Theta_k = |\arg r_{(2n+1)^k}(e^{i\Theta};\Theta)|,
\]
for each $k \ge 0$.  On the other hand, the composition law~(\ref{Rcomp}) implies that in the iteration~(\ref{zolopd1}-\ref{zolopd2}), we have
\[
Y_k = \widehat{R}_{(2n+1)^k}(B;\ell), \quad \ell_k = \widehat{R}_{(2n+1)^k}(\ell;\ell),
\]
for each $k \ge 0$.
Thus, by~(\ref{rplusrinv}),
\begin{align*}
\frac{1}{2}(X_k+X_k^*) 
&= \frac{1}{2}(X_k+X_k^{-1}) \\
&=  \frac{1}{2}\left(r_{(2n+1)^k}(A;\Theta) + r_{(2n+1)^k}(A;\Theta)^{-1}\right) \\
&= \widehat{R}_{(2n+1)^k}((A+A^{-1})/2;\cos\Theta) \\
&= \widehat{R}_{(2n+1)^k}(B;\ell) \\
&= Y_k.
\end{align*}
Also, by Theorem~\ref{thm:realpart},
\[
\cos\Theta_k = \Re e^{i\Theta_k} = \Re r_{(2n+1)^k}(e^{i\Theta};\Theta) = \widehat{R}_{(2n+1)^k}(\Re e^{i\Theta};\cos\Theta) = \widehat{R}_{(2n+1)^k}(\ell;\ell) = \ell_k.
\]
\end{proof}

In the case that $\Theta=0$, the above result implies a connection between the diagonal Pad\'e iterations~(\ref{pade}) and~(\ref{padepd}).

\begin{corollary}
Let $A$ be a unitary matrix with no eigenvalues equal to $\pm i$, and let $n \in \mathbb{N}$.  If $B = (A+A^*)/2$, then the diagonal Pad\'e iterations~(\ref{pade}) and~(\ref{padepd}) generate sequences satisfying
\[
Y_k = \frac{1}{2}(X_k+X_k^*)
\]
for every $k \ge 0$.
\end{corollary}

\subsection{Implementation}

To implement the $k$th step of the iteration~(\ref{zolo1}-\ref{zolo2}), one must compute products of unitary matrices of the form
\begin{equation} \label{XaX}
V_j  = (X_k^2 + a_j I) (I+a_j X_k^2)^{-1} = (X_k + a_j X_k^*) (X_k^*+a_j X_k)^{-1}, \quad j=1,2,\dots,n,
\end{equation}
where $X_k$ is unitary.  The following lemma describes a method for computing~(\ref{XaX}) that is guaranteed to produce a matrix that is unitary to machine precision.

\begin{lemma}
Let $B \in \mathbb{C}^{m \times m}$ be a nonsingular normal matrix.  Let $Q_1 R_1 = B$ and $Q_2 R_2 = B^*$ be the QR factorizations of $B$ and $B^*$, respectively.  Then
\[
B B^{-*} = Q_1 Q_2^*.
\]
\end{lemma}
\begin{proof}
Since $R_1$ is the Cholesky factor of $B^*B$ and $R_2$ is the Cholesky factor of $BB^* = B^*B$, we have $R_1=R_2$.  Hence, $BB^{-*} = Q_1 R_1 R_2^{-1} Q_2^* = Q_1 Q_2^*$.
\end{proof}

Once~(\ref{XaX}) has been computed for each $j$, one must decide in what order to multiply the matrices $V_1, V_2, \dots, V_n$, and $X_k$.  Our numerical experience suggests that this decision has a strong influence on the backward stability of the algorithm.  We find that the choice
\begin{equation} \label{Xkp1}
X_{k+1} = \frac{1}{2}(X_k V_1 V_2 \cdots V_n + V_n V_{n-1} \cdots V_1 X_k)
\end{equation}
is preferable to, for instance, $X_{k+1} = X_k V_1 V_2 \cdots V_n$ or $X_{k+1} = V_n V_{n-1} \cdots V_1 X_k$.  This choice appears to guarantee that $\|X_k A - A X_k\|=O(u)$ for each $k$, which is essential for backward stability; see Lemma~\ref{lemma:backwardstability} for details.  A proof that $\|X_k A - A X_k\|=O(u)$ when~(\ref{Xkp1}) is used remains an open problem.

\paragraph{Termination}
We must also decide how to terminate the iteration.  
Here we suggest terminating slightly early and applying two post-processing steps---symmetrization followed by one step of the Newton-Schulz iteration~\cite[Equation 8.20]{higham2008functions} for the polar decomposition---to ensure that the computed matrix $\widehat{S} \approx \sign(A)$ is Hermitian and unitary to machine precision.  These post-processing steps have the following effect.
Let $\{\sigma_j \cos\theta_j + i\sin\theta_j\}_{j=1}^m$ be the eigenvalues of $X_k$, where $\sigma_j \in \{-1,1\}$ and $|\theta_j|<\pi/2$ for each $j$. Then 
\begin{equation} \label{symmetrize}
Y = \frac{1}{2}(X_k + X_k^*)
\end{equation}
has eigenvalues $\{\sigma_j \cos\theta_j \}_{j=1}^m$, and
\begin{equation} \label{newtonschulz}
Z = \frac{1}{2} Y(3I-Y^*Y) = \frac{1}{2} Y(3I-Y^2)
\end{equation}
has eigenvalues $\{ \frac{1}{2}\sigma_j \cos\theta_j (3 - \cos^2\theta_j) \}_{j=1}^m$.  For small $\theta_j$, we have
\[
\frac{1}{2}\sigma_j \cos\theta_j (3 - \cos^2\theta_j) = \sigma_j \left(1-\frac{3}{8}\theta_j^4\right) + O(\theta_j^6).
\]
This number will lie within a tolerance $\delta$ of $\pm 1$ if
\begin{equation} \label{thetaconverged}
\theta_j \lesssim \left( \frac{8\delta}{3} \right)^{1/4}.
\end{equation}

The above calculations suggest the following termination criterion.  Since the eigenvalues of $X_k-X_k^*$ are $\{2i\sin\theta_j\}_{j=1}^m \approx \{2i\theta_j\}_{j=1}^m$, we terminate the iteration and carry out the post-processing steps~(\ref{symmetrize}-\ref{newtonschulz}) as soon as
\[
\|X_k-X_k^*\| \le 2 \left( \frac{8\delta}{3} \right)^{1/4}.
\]
Note that since the Frobenius norm $\|\cdot\|_F$ is an upper bound for the $2$-norm $\|\cdot\|$, we may safely replace $\|X_k-X_k^*\|$ by $\|X_k-X_k^*\|_F$ in the criterion above.  If desired, a second symmetrization can be performed after the Newton-Schulz step.  This has virtually no effect on the eigenvalues' distance to $\pm 1$, but it may be desirable if an exactly Hermitian matrix is sought.

\paragraph{Spectral angle}
Let us also mention how to determine $\Theta$ so that $\Lambda(A) \subset \mathbb{S}_\Theta$.  We hereafter refer to the smallest such $\Theta$ as the \emph{spectral angle} of $A$, denoted $\Theta(A)$.  A simple heuristic is to estimate the eigenvalues $\lambda_+$ and $\lambda_-$ of $A$ that lie closest to $i$ and $-i$, respectively.  Then one can set
\[
\Theta = \max\{ \pi/2 - |\arg(i\lambda_-)|, |\arg(i\lambda_+)| - \pi/2 \}.
\]
In practice, it is not necessary to determine the spectral angle of $A$ precisely.
Our experience suggests that underestimates and overestimates of $\Theta$ can be used without significant harm, unless $\Theta$ is very close to $\pi/2$.

\paragraph{Spectral angles close to $\pi/2$}
There are a few delicate numerical issues that arise when the spectral angle of $A$ is close to $\pi/2$.  First, as noted in~\cite[Section 4.3]{nakatsukasa2016computing}, the built-in MATLAB functions \verb$ellipj$ and \verb$ellipke$ cannot be used to reliably compute $\sn(\cdot,\ell')$, $\cn(\cdot,\ell')$, $\dn(\cdot,\ell')$, and $K(\ell')$ when $\Theta = \arccos\ell$ is close to $\pi/2$.  Instead, the code described in~\cite[Section 4.3]{nakatsukasa2016computing} is preferred.  In addition, the lowest-order iteration ($n=1$) appears to be more reliable than the higher-order iterations when $\Theta>\pi/2-u^{1/2}$, so we advocate using the lowest-order iteration until $\Theta_k$ falls below $\pi/2-u^{1/2}$ (recall that $u=2^{-53}$ denotes the unit roundoff).  Typically this takes two or fewer iterations, after which one can switch to a higher-order iteration if desired.

To implement the lowest-order iteration ($n=1$) when $\Theta>\pi/2-u^{1/2}$, we have found the following heuristic to be useful for ensuring rapid convergence.  If, at the $k$th iteration, $\Theta_k$ lies above $\pi/2-u^{1/2}$, we compute $\Theta_{k+1}$ as $\Theta_{k+1}=\Theta(X_{k+1})$ (the spectral angle of $X_{k+1}$) rather than via~(\ref{zolo2}).  This tends to speed up the iteration.  
To improve stability, we have also found it prudent to replace $\Theta_k$ by $\pi/2-10u$ if $\Theta_k > \pi/2-10u$.

A summary of our proposed algorithm for computing the unitary sign decomposition is presented in Algorithm~\ref{alg:zolosign}.

\begin{algorithm}
\caption{Order-$(2n+1)$ iteration for the unitary sign decomposition\newline
\textit{Inputs}: Unitary matrix $A \in \mathbb{C}^{m \times m}$, tolerance $\delta>0$, degree $n \in \mathbb{N}$\newline
\textit{Outputs}: Matrices $S,N \in \mathbb{C}^{m \times m}$ satisfying~(\ref{unitarysign})}
\label{alg:zolosign}
\begin{algorithmic}[1]
\STATE{$\Theta_0 = \min\{\Theta(A),\pi/2-10u$\}} \label{line:theta0}
\STATE{$X_0=A$, $n_0=n$, $k=0$}
\WHILE{$\|X_k-X_k^*\|_F > 2(8\delta/3)^{1/4}$}
\LINEIFELSE{$\Theta_k>\pi/2-u^{1/2}$}{$n=1$}{$n=n_0$}
\STATE{$Y=X_k$, $Z=X_k$}
\FOR{$j=1$ \TO $n$}
\STATE{$Q_1 R_1 = X_k+a_j(\Theta_k) X_k^*$ (QR factorization)}
\STATE{$Q_2 R_2 = X_k^*+a_j(\Theta_k) X_k$ (QR factorization)}
\STATE{$Y= Y Q_1 Q_2^* $}
\STATE{$Z = Q_1 Q_2^* Z$}
\ENDFOR
\STATE{$X_{k+1} = \frac{1}{2}(Y+Z)$}
\IF{$\Theta_k>\pi/2-u^{1/2}$}
\STATE{$\Theta_{k+1} = \min\{\Theta(X_{k+1}),\pi/2-10u\}$}
\ELSE
\STATE{$\Theta_{k+1} = |\arg r_{2n+1}(e^{i\Theta_k};\Theta_k)|$}
\ENDIF
\STATE{$k = k+1$}
\ENDWHILE
\STATE{$S = (X_k+X_k^*)/2$}
\STATE{$S = S(3I-S^2)/2$}
\STATE{$S = (S+S^*)/2$}
\STATE{$N = S A$}
\RETURN $S$, $N$
\end{algorithmic} 
\end{algorithm}

\subsection{Backward Stability}

We now discuss how some of the choices made above are inspired by backward stability considerations.

We first address a remark that was made in the footnote of this paper's introduction concerning the list of backward errors~(\ref{backwarderrors}).  At first glance, this list may appear to be incomplete because the norm of $\widehat{N}\widehat{S}-\widehat{S}\widehat{N}$ is absent.  
The following lemma shows that if $\widehat{S}$ and $\widehat{N}$ are well-conditioned matrices and $\|\widehat{N}^2-A^2\|$, $\|A-\widehat{S}\widehat{N}\|$, and $\|\widehat{S}^2-I\|$ are small, then $\|\widehat{N}\widehat{S}-\widehat{S}\widehat{N}\|$ is automatically small as well.

\begin{lemma} \label{lemma:SNcommute}
Let $A \in \mathbb{C}^{m \times m}$ be a unitary matrix.  For any invertible matrices $\widehat{S},\widehat{N} \in \mathbb{C}^{m \times m}$, we have
\begin{equation*}
\|\widehat{N}\widehat{S}-\widehat{S}\widehat{N}\| \le \left( \|\widehat{N}^2-A^2\| + (1 + \|\widehat{S}\|\|\widehat{N}\|) \|A-\widehat{S}\widehat{N}\| + \|\widehat{N}\|^2 \|\widehat{S}^2-I\| \right) \|\widehat{N}^{-1}\| \|\widehat{S}^{-1}\|.
\end{equation*}
\end{lemma}
\begin{proof}
This follows from the identity
\[
(\widehat{N}\widehat{S}-\widehat{S}\widehat{N})\widehat{S}\widehat{N} = \widehat{N}^2-A^2 + A(A-\widehat{S}\widehat{N}) + (A-\widehat{S}\widehat{N})\widehat{S}\widehat{N} + \widehat{N}(\widehat{S}^2-I)\widehat{N}.
\]
\end{proof}

The next lemma shows that in order to achieve backward stability, it is prudent to compute a Hermitian matrix $\widehat{S}$ such that $\|\widehat{S}^2-I\|$ and $\|A\widehat{S}-\widehat{S}A\|$ are small, and then set $\widehat{N} = \widehat{S}A$.  This highlights the importance of ensuring the smallness of $\|AX_k-X_kA\|$ in Algorithm~\ref{alg:zolosign}.

\begin{lemma} \label{lemma:backwardstability}
Let $A \in \mathbb{C}^{m \times m}$ be a unitary matrix, let $\widehat{S}$ be an invertible Hermitian matrix, and  let $\widehat{N} = \widehat{S}A$.  Then
\begin{align}
\|\widehat{N}^*\widehat{N}-I\| &\le \|\widehat{S}^2-I\|, \label{ineq1} \\
\|A-\widehat{S}\widehat{N}\| &\le \|\widehat{S}^2-I\|, \label{ineq2} \\
\|\widehat{N}^2-A^2\| &\le \|\widehat{S}\| \|A\widehat{S}-\widehat{S}A\| + \|\widehat{S}^2-I\|, \label{ineq3} \\
\|\widehat{N}\widehat{S}-\widehat{S}\widehat{N}\| &\le \|\widehat{S}\| \|A\widehat{S}-\widehat{S}A\|. \label{ineq4}
\end{align}
\end{lemma}
\begin{proof}
Since $A^*A=I$, $\widehat{N}=\widehat{S}A$, and $\widehat{S}=\widehat{S}^*$, we have
\[
\widehat{N}^*\widehat{N}-I = A^*\widehat{S}^2 A - I = A^* (\widehat{S}^2-I)A.
\]
Taking the norm of both sides proves~(\ref{ineq1}).  Similarly, the equalities
\begin{align*}
A-\widehat{S}\widehat{N} &= (I-\widehat{S}^2)A, \\
\widehat{N}^2-A^2 &= \widehat{S}A\widehat{S}A - A^2 = \widehat{S}(A\widehat{S}-\widehat{S}A)A + (\widehat{S}^2-I)A^2, \\
\widehat{N}\widehat{S}-\widehat{S}\widehat{N} &= \widehat{S}(A\widehat{S}-\widehat{S}A)
\end{align*}
yield~(\ref{ineq2}-\ref{ineq4}).
\end{proof}

\section{A Spectral Divide-and-Conquer Algorithm for the Unitary Eigendecomposition} \label{sec:eig}

The iteration we have proposed for computing the unitary sign decomposition can be used to construct a spectral divide-and-conquer algorithm for the unitary eigendecomposition, following~\cite{nakatsukasa2013stable,nakatsukasa2016computing}.  The idea is as follows.  Given a unitary matrix $A \in \mathbb{C}^{m \times m}$, we scale $A$ by a complex number $e^{i\phi}$ so that roughly half (say, $m_1$) of the eigenvalues of $e^{i\phi}A$ lie in the right half of the complex plane, and roughly half (say, $m_2$) lie in the left half of complex plane.  We then compute $S = \sign(e^{i\phi}A)$ using Algorithm~\ref{alg:zolosign}.  The matrix $P=(I+S)/2$ is a spectral projector onto the invariant subspace $\mathcal{V}_+$ of $e^{i\phi}A$ associated with the eigenvalues of $e^{i\phi}A$ having positive real part.   Using subspace iteration, we can compute orthonormal bases $U_1 \in \mathbb{C}^{m \times m_1}$ and $U_2 \in \mathbb{C}^{m \times m_2}$ (where $m_1+m_2=m$) for $\mathcal{V}_+$ and its orthogonal complement.  Then
\[
\begin{pmatrix} U_1^* \\  U_2^* \end{pmatrix} A \begin{pmatrix} U_1 & U_2 \end{pmatrix} 
=
\begin{pmatrix} A_1 & 0 \\ 0 & A_2 \end{pmatrix}
\]
is block diagonal, so we can recurse to find eigendecompositions $A_1 = V_1 \Lambda_1 V_1^*$ and $A_2 = V_2 \Lambda_2 V_2^*$.  The eigendecomposition of $A$ is then $A = V\Lambda V^*$, where
\[
V = \begin{pmatrix} U_1 V_1 & U_2 V_2 \end{pmatrix}
\]
and
\[
\Lambda = \begin{pmatrix} \Lambda_1 & 0 \\ 0 & \Lambda_2 \end{pmatrix}.
\]

Since every eigenvalue of $P$ is either $0$ and $1$, subspace iteration with $P$ typically converges in one iteration, or, in rare cases, two.  To choose the scalar $e^{i\phi}$, a simple heuristic is to compute the median $\mu$ of the arguments of the diagonal entries of $A$ and set $\phi=\pi/2-\mu$.  When $A$ is nearly diagonal, this has the effect of centering the eigenvalues around $i$.

A summary of the algorithm just described is presented in Algorithm~\ref{alg:eig}.

\begin{algorithm}
\caption{Divide-and-conquer algorithm for the unitary eigendecomposition\newline
\textit{Inputs}: Unitary matrix $A \in \mathbb{C}^{m \times m}$ \newline
\textit{Outputs}: Matrices $V,\Lambda \in \mathbb{C}^{m \times m}$ satisfying $V\Lambda V^* = A$, $V^*V = I$, and $\Lambda$ diagonal}
\label{alg:eig}
\begin{algorithmic}[1]
\STATE{$\phi = \frac{\pi}{2} - \operatorname{median} \{\arg A_{11},\dots,\arg A_{mm}\}$}
\STATE{$S = \sign(e^{i\phi} A)$} \label{line:sign}
\STATE{$P = (I+S)/2$}
\STATE{Use subspace iteration to compute orthonormal bases $U_1 \in \mathbb{C}^{m \times m_1}$ and $U_2 \in \mathbb{C}^{m \times m_2}$ for the 0- and 1-eigenspaces of $P$.}
\STATE{$A_1 = U_1^* A U_1$, $A_2 = U_2^* A U_2$}
\STATE{Recurse to find eigendecompositions $V_1 \Lambda_1 V_1^* = A_1$ and $V_2 \Lambda_2 V_2^* = A_2$.}
\STATE{$V = \begin{pmatrix} U_1 V_1 & U_2 V_2 \end{pmatrix}$}
\STATE{$\Lambda = \begin{pmatrix} \Lambda_1 & 0 \\ 0 & \Lambda_2 \end{pmatrix}$}
\RETURN $V$, $\Lambda$
\end{algorithmic} 
\end{algorithm}

\section{Numerical Examples} \label{sec:numerical}

In this section, we study the iteration~(\ref{zolo1}-\ref{zolo2}) numerically, and we test Algorithms~\ref{alg:zolosign} and~\ref{alg:eig} on a collection of unitary matrices.

\subsection{Scalar Iteration}

\begin{table}
\centering
\begin{tabularx}{\linewidth}{ Y*{11}{Y} }
 & \multicolumn{11}{c}{$\frac{\pi}{2}-\Theta$} \\
$n$ & 1.5 & 1 & 0.5 & $10^{-2}$ & $10^{-4}$ & $10^{-6}$ & $10^{-8}$ & $10^{-10}$ & $10^{-12}$ & $10^{-14}$ & $10^{-16}$ \\
\cmidrule(lr){1-1}
\cmidrule(lr){2-12}
 1 & 1 & 2 & 2 & 3 & 4 & 4 & 5 & 5 & 5 & 5 & 5 \\
 2 & 1 & 2 & 2 & 3 & 3 & 3 & 3 & 3 & 3 & 4 & 4 \\
 3 & 1 & 1 & 2 & 2 & 2 & 3 & 3 & 3 & 3 & 3 & 3 \\
 4 & 1 & 1 & 1 & 2 & 2 & 2 & 3 & 3 & 3 & 3 & 3 \\
 5 & 1 & 1 & 1 & 2 & 2 & 2 & 2 & 2 & 3 & 3 & 3 \\
 6 & 1 & 1 & 1 & 2 & 2 & 2 & 2 & 2 & 2 & 2 & 2 \\
 7 & 1 & 1 & 1 & 2 & 2 & 2 & 2 & 2 & 2 & 2 & 2 \\
 8 & 1 & 1 & 1 & 2 & 2 & 2 & 2 & 2 & 2 & 2 & 2 \\
\end{tabularx}
\caption{Smallest integer $k$ for which $4\rho(\Theta)^{-(2n+1)^k} \le (8\delta/3)^{1/4}$, where $\delta = 10^{-16}$, for various values of $n$ and $\Theta$.}
\label{tab:iter}
\end{table}

\begin{table}
\centering
\begin{tabularx}{\linewidth}{ Y*{11}{Y} }
 & \multicolumn{11}{c}{$\frac{\pi}{2}-\Theta$} \\
$n$ & 1.5 & 1 & 0.5 & $10^{-2}$ & $10^{-4}$ & $10^{-6}$ & $10^{-8}$ & $10^{-10}$ & $10^{-12}$ & $10^{-14}$ & $10^{-16}$ \\
\cmidrule(lr){1-1}
\cmidrule(lr){2-12}
 1 & 1 & 2 & 3 & 7 & 11 & 15 & 19 & 24 & 28 & 32 & 37 \\
 2 & 1 & 2 & 2 & 5 & 8 & 10 & 13 & 16 & 19 & 22 & 25 \\
 3 & 1 & 2 & 2 & 4 & 6 & 9 & 11 & 13 & 16 & 18 & 21 \\
 4 & 1 & 1 & 2 & 4 & 6 & 8 & 10 & 12 & 14 & 16 & 19 \\
 5 & 1 & 1 & 2 & 3 & 5 & 7 & 9 & 11 & 13 & 15 & 17 \\
 6 & 1 & 1 & 2 & 3 & 5 & 7 & 9 & 10 & 12 & 14 & 16 \\
 7 & 1 & 1 & 2 & 3 & 5 & 6 & 8 & 10 & 12 & 13 & 15 \\
 8 & 1 & 1 & 2 & 3 & 5 & 6 & 8 & 9 & 11 & 13 & 14 
\end{tabularx}
\caption{Smallest integer $k$ for which $|r_{(2n+1)^k}(e^{i\Theta};0)-1| \le (8\delta/3)^{1/4}$, where $\delta = 10^{-16}$, for various values of $n$ and $\Theta$.}
\label{tab:iterpade}
\end{table}

To understand how rapidly the iteration~(\ref{zolo1}-\ref{zolo2}) can be expected to converge, let us study the upper bound~(\ref{errorbound}).  Table~\ref{tab:iter} reports the smallest integer $k$ for which $4\rho(\Theta)^{-(2n+1)^k}$ falls below the number $(8\delta/3)^{1/4}$ appearing in the convergence criterion~(\ref{thetaconverged}).  Here, we took $\delta = 10^{-16}$ and considered various choices of $n$ and $\Theta$.  The integer $k$ so computed provides an estimate for the number of iterations one can expect~(\ref{zolo1}-\ref{zolo2}) to take to converge to $\sign(A)$ if $A$ has spectrum contained in $\mathbb{S}_\Theta$.

For comparison, we computed the number of iterations needed for the scalar Pad\'e iteration
\[
z_{k+1} = r_{2n+1}(z_k;0) = z_k p_n(z_k^2)
\]
to converge to $\sign z_0$, starting from $z_0 = e^{i\Theta}$.  The results, reported in Table~\ref{tab:iterpade}, show that the Pad\'e iterations take significantly longer to converge if $\Theta$ is close to $\pi/2$.  This suggests the matrix Pad\'e iteration~(\ref{pade}) will require a large number of iterations to converge to $\sign(A)$ if the spectral angle $\Theta(A)$ is close to $\pi/2$.

\subsection{Matrix Iteration}

To test Algorithm~\ref{alg:zolosign}, we computed the sign decomposition of four unitary matrices:
\begin{enumerate}
\item \label{mat1} A matrix sampled randomly from the Haar measure on the $m \times m$ unitary group.
\item  \label{mat2} \verb$A = gallery('orthog',m,3)$.  This is the $m$-point discrete Fourier transform matrix with entries $A_{jk} = e^{2\pi i (j-1)(k-1)/m}  / \sqrt{m}$.  Its eigenvalues are $1,-1,i,-i$.  The spectrum of the floating point representation of $A$ therefore includes $O(u)$-perturbations of $\pm i$, posing a challenge to numerical algorithms for the unitary sign decomposition.
\item \label{mat3} \verb$A = circshift(eye(m),1)$.  This is a permutation matrix with eigenvalues $e^{2\pi i j/m}$, $m=1,2,\dots,m$.  For even $m$, the spectrum of $A$ includes $\pm i$.  The same is true of the floating point representation of $A$, since the entries of $A$ are integers.
\item \label{mat4} \verb$A = gallery('orthog',m,-2)$ (with columns normalized).  The entries of $A$ (prior to normalizing columns) are $A_{jk} = \cos((k-1/2)(j-1)\pi/m)$.  The spectrum of $A$ is clustered near $\pm 1$, making its sign decomposition somewhat easy to compute iteratively.
\end{enumerate}

\medskip

In our numerical experiment, we used $m=100$.  The computed spectral angles for the matrices above were $\pi/2-\Theta(A) = 0.026$, $4.4 \times 10^{-16}$, $0$, and $0.95$, respectively.

On each of the matrices above, we compared 10 algorithms: 
\begin{itemize}
\item Algorithm~\ref{alg:zolosign} with $n=1,4,8$.
\item The diagonal Pad\'e iteration~(\ref{pade}) with $n=1,4,8$.  We implemented this by running Algorithm~\ref{alg:zolosign} with line~\ref{line:theta0} replaced by $\Theta_0=0$.
\item Three algorithms that compute the unitary factor $S$ in the polar decomposition of $B=(A+A^*)/2$.  The first uses the Newton iteration with $1,\infty$-norm scaling, as described in~\cite[Section 8.6]{higham2008functions} and implemented in~\cite{Higham:MCT}.  The second uses the Zolo-pd algorithm from~\cite{nakatsukasa2016computing}.  The third computes $S$ as $S=UV^*$, where $B=U\Sigma V^*$ is the SVD of $B$.  In all three cases, we applied post-processing to $S$ ($S=(S+S^*)/2$, followed by $S=S(3I-S^2)/2$, followed by $S=(S+S^*)/2$) and set $N=SA$.
\item A direct method: computing the eigendecomposition $A=V\Lambda V^*$ of $A$ and setting $S = V\sign(\Lambda)V^*$.  We computed the eigendecomposition by using the MATLAB command \verb$schur(A,'complex')$ and setting the off-diagonal entries of the triangular factor to zero.  We applied post-processing to $S$ ($S=S(3I-S^2)/2$ followed by $S=(S+S^*)/2$) and set $N=SA$.
\end{itemize}

\medskip

The results of the tests are reported in Table~\ref{tab:sign}.  All of the algorithms under consideration performed in a backward stable way on the first and fourth matrices.  On the second and third matrices (\verb$gallery('orthog',m,3)$ and \verb$circshift(eye(m),1)$), only the direct method and the structure-preserving iterations (Algorithm~\ref{alg:zolosign} and the Pad\'e iteration~(\ref{pade})) exhibited backward stability.  Among the structure-preserving iterations, Algorithm~\ref{alg:zolosign} consistently converged more quickly than the Pad\'e iteration~(\ref{pade}) for each degree $n$.  The reduction in iteration count was particularly noticeable for \verb$gallery('orthog',m,3)$ and \verb$circshift(eye(m),1)$.

\begin{table}[t]
\centering
\begin{filecontents*}{sign.dat}
   3.0000000e+00   1.3127648e-15   1.0529381e-15   1.8085866e-15   2.4365588e-15   0.0000000e+00
   2.0000000e+00   1.2214788e-15   9.4340109e-16   1.8577851e-15   3.8890325e-15   0.0000000e+00
   2.0000000e+00   1.1868139e-15   1.0247180e-15   1.7579170e-15   4.9236228e-15   0.0000000e+00
   6.0000000e+00   1.2078036e-15   9.4067684e-16   1.7691431e-15   2.7175149e-15   0.0000000e+00
   3.0000000e+00   1.2103349e-15   9.3774814e-16   2.0595934e-15   4.7973292e-15   0.0000000e+00
   3.0000000e+00   1.2078090e-15   9.7198718e-16   1.6551114e-15   6.1560650e-15   0.0000000e+00
   7.0000000e+00   1.2122182e-15   1.0062068e-15   1.6534555e-15   2.7530909e-14   0.0000000e+00
   2.0000000e+00   1.0480053e-15   6.4187905e-16   1.6171863e-15   2.4771314e-15   0.0000000e+00
   0.0000000e+00   1.1651174e-15   9.4567716e-16   1.7399948e-15   7.3539072e-15   0.0000000e+00
   0.0000000e+00   1.2414826e-15   1.0552697e-15   1.8192189e-15   1.1262979e-14   0.0000000e+00
   6.0000000e+00   1.1893666e-15   9.7720275e-16   2.2760767e-15   3.3018705e-15   0.0000000e+00
   4.0000000e+00   1.2163615e-15   1.0074135e-15   2.3010580e-15   1.0823923e-14   2.0677904e-15
   4.0000000e+00   1.2028533e-15   9.8035275e-16   1.8360434e-15   7.6405171e-15   9.9920072e-16
   3.4000000e+01   1.2990476e-15   1.3478085e-15   1.8323212e-15   9.0907536e-15   0.0000000e+00
   1.7000000e+01   1.2659551e-15   1.1922011e-15   2.0165344e-15   6.3402580e-14   1.5699247e-16
   1.4000000e+01   1.2727126e-15   1.1343007e-15   2.2872075e-15   7.7383126e-14   3.6637360e-15
   8.0000000e+00   1.1763585e-15   9.1938782e-16   1.6546533e-15   3.4122261e-01   1.7061131e-01
   2.0000000e+00   1.2195775e-15   6.4748189e-16   2.4786129e-15   1.9992204e-01   9.9961018e-02
   0.0000000e+00   1.8217944e-02   1.8217944e-02   1.8217944e-02   3.5820764e-01   1.7877602e-01
   0.0000000e+00   1.2100710e-15   1.1378225e-15   1.7903259e-15   8.5284739e-15   0.0000000e+00
   6.0000000e+00   1.1511011e-15   9.6279960e-16   1.0900951e-15   4.4246331e-15   0.0000000e+00
   4.0000000e+00   1.2844290e-15   8.6921282e-16   1.1817284e-15   6.4217017e-15   0.0000000e+00
   4.0000000e+00   1.0992165e-15   9.4336481e-16   1.0434671e-15   5.5265477e-15   0.0000000e+00
   3.7000000e+01   4.0648057e-15   4.0759975e-15   4.0622101e-15   7.9878611e-15   5.4123372e-16
   1.9000000e+01   1.6360112e-15   1.6358436e-15   1.6361675e-15   5.0280676e-14   5.5511151e-16
   1.4000000e+01   1.8179828e-15   1.8181028e-15   1.8173364e-15   1.0971284e-13   7.2164497e-16
   7.0000000e+00   7.1302007e-16   6.4234459e-16   6.8106523e-16   2.0000000e+00   1.0000000e+00
   2.0000000e+00   7.0133820e-06   7.0133820e-06   7.0133820e-06   2.0000000e+00   1.0000000e+00
   0.0000000e+00   2.3145819e-15   1.5891784e-15   2.1442433e-15   2.0000000e+00   1.0000000e+00
   0.0000000e+00   9.9537331e-16   9.9537331e-16   9.9710984e-16   1.0511078e-14   0.0000000e+00
   2.0000000e+00   1.5126998e-15   1.2378292e-15   2.0141919e-15   2.5189384e-15   0.0000000e+00
   1.0000000e+00   1.3179507e-15   1.2404878e-15   1.8710048e-15   3.0407441e-15   0.0000000e+00
   1.0000000e+00   1.3162823e-15   9.6186823e-16   2.1396163e-15   3.7662577e-15   0.0000000e+00
   3.0000000e+00   1.4787871e-15   1.0080126e-15   2.2248919e-15   2.5258622e-15   0.0000000e+00
   2.0000000e+00   1.2582087e-15   9.9585622e-16   2.1891941e-15   3.1058971e-15   0.0000000e+00
   1.0000000e+00   1.3126408e-15   1.0147683e-15   2.0215065e-15   3.7539231e-15   0.0000000e+00
   4.0000000e+00   1.2054957e-15   8.3650821e-16   2.1526058e-15   3.7905959e-15   0.0000000e+00
   1.0000000e+00   1.1514422e-15   6.3951488e-16   1.9615443e-15   2.3134178e-15   0.0000000e+00
   0.0000000e+00   1.2392777e-15   9.9812717e-16   2.0671135e-15   8.0453177e-15   0.0000000e+00
   0.0000000e+00   1.2684461e-15   1.0210629e-15   1.9773736e-15   9.4731762e-15   0.0000000e+00
\end{filecontents*}
\pgfplotstabletypeset[
every head row/.style={
after row=\midrule},
every nth row={10}{before row=\midrule},
columns={leftcol2,0,1,2,3,4,5},
create on use/leftcol2/.style={create col/set list={
Alg.~\ref{alg:zolosign} ($n=1$),Alg.~\ref{alg:zolosign} ($n=4$),Alg.~\ref{alg:zolosign} ($n=8$), Pad\'e ($n=1$), Pad\'e ($n=4$), Pad\'e ($n=8$), Polar (Newton), Polar (Zolo-pd), Polar (SVD), Direct,
Alg.~\ref{alg:zolosign} ($n=1$),Alg.~\ref{alg:zolosign} ($n=4$),Alg.~\ref{alg:zolosign} ($n=8$), Pad\'e ($n=1$), Pad\'e ($n=4$), Pad\'e ($n=8$), Polar (Newton), Polar (Zolo-pd), Polar (SVD), Direct,
Alg.~\ref{alg:zolosign} ($n=1$),Alg.~\ref{alg:zolosign} ($n=4$),Alg.~\ref{alg:zolosign} ($n=8$), Pad\'e ($n=1$), Pad\'e ($n=4$), Pad\'e ($n=8$), Polar (Newton), Polar (Zolo-pd), Polar (SVD), Direct,
Alg.~\ref{alg:zolosign} ($n=1$),Alg.~\ref{alg:zolosign} ($n=4$),Alg.~\ref{alg:zolosign} ($n=8$), Pad\'e ($n=1$), Pad\'e ($n=4$), Pad\'e ($n=8$), Polar (Newton), Polar (Zolo-pd), Polar (SVD), Direct}},
columns/leftcol/.style={string type,column type/.add={}{},column name={$\frac{\pi}{2}-\Theta(A)$}},
columns/leftcol/.style={string type,column type/.add={}{},column name={$\pi/2-\Theta(A)$}},
columns/leftcol2/.style={string type,column type/.add={}{},column name={Algorithm}},
columns/0/.style={column type/.add={}{},column name={$k$}}, 
columns/1/.style={sci,sci e,sci zerofill,precision=1,column type/.add={}{},column name={$\|A-\widehat{S}\widehat{N}\|$}}, 
columns/2/.style={sci,sci e,sci zerofill,precision=1,column type/.add={}{},column name={$\|\widehat{S}^2-I\|$}}, 
columns/3/.style={sci,sci e,sci zerofill,precision=1,column type/.add={}{},column name={$\|\widehat{N}^*\widehat{N}-I\|$}}, 
columns/4/.style={sci,sci e,sci zerofill,precision=1,column type/.add={}{},column name={$\|\widehat{N}^2-A^2\|$}}, 
columns/5/.style={sci,sci e,sci zerofill,precision=1,column type/.add={}{},column name={$\mu(\widehat{N})$}}
]
{sign.dat}
\caption{Performance of algorithms for computing the unitary sign decomposition of the matrices~\ref{mat1}-\ref{mat4}.  The table reports the iteration count $k$ and backward errors $\|A-\widehat{S}\widehat{N}\|$, $\|\widehat{S}^2-I\|$, $\|\widehat{N}^*\widehat{N}-I\|$, $\|\widehat{N}^2-A^2\|$, $\mu(\widehat{N})=\max\{0,-\min_{\lambda \in \Lambda(\widehat{N})} \Re\lambda\}$ for each algorithm.}
\label{tab:sign}
\end{table}

\subsection{Unitary eigendecomposition}

Next, we tested our spectral divide-and-conquer algorithm~\ref{alg:eig} on the same four matrices.  We implemented line~\ref{line:sign} of Algorithm~\ref{alg:eig} in nine different ways, namely, by using the nine indirect methods considered in the previous experiment.  We compared the results with the following direct method: \verb$[V,Lambda]=schur(A,'complex'); Lambda = diag(diag(Lambda))$.  The results are reported in Table~\ref{tab:eig}.  

All of the algorithms under consideration performed in a backward stable way on the first, second, and fourth matrices.  On the third matrix \verb$circshift(eye(m),1)$, the algorithms that used Zolo-pd and the SVD did not.  Curiously, the algorithm that used the Newton iteration succeeded, but this is an anomaly.  Changing \verb$circshift(eye(m),1)$ to \verb$circshift(eye(m),1)+eps*randn(m)$ leads to a backward error $\|A-\widehat{V}\widehat{\Lambda}\widehat{V}^*\|$ close to 0.1 for the Newton-based algorithm, and it has a negligible effect on the other algorithms' backward errors.

\begin{table}[t]
\hspace{-0.25in}
\begin{filecontents*}{eig12.dat}
   4.1289810e-15   3.3016713e-15
   4.9800814e-15   3.7928904e-15
   4.7717379e-15   3.1709822e-15
   5.1633259e-15   3.8685026e-15
   5.0605607e-15   3.3504431e-15
   5.7240247e-15   3.7563484e-15
   1.2945336e-14   3.2436766e-15
   5.1659007e-15   3.5291869e-15
   4.3766706e-15   3.2923952e-15
   1.5440496e-14   1.1669319e-14
   5.9719528e-15   2.9464351e-15
   5.9422345e-15   2.8206186e-15
   6.2742395e-15   2.7890555e-15
   5.8626693e-15   3.2599146e-15
   6.1984514e-15   3.1329662e-15
   6.5827302e-15   2.6706638e-15
   1.2710028e-14   2.8479278e-15
   6.2782478e-15   2.5921034e-15
   8.5339376e-15   2.6400699e-15
   1.6712472e-14   1.1291287e-14
\end{filecontents*}
\pgfplotstabletypeset[
every head row/.style={after row=\midrule},
every nth row={10}{before row=\midrule},
columns={leftcol2,0,1},
create on use/leftcol2/.style={create col/set list={
Alg.~\ref{alg:zolosign} ($n=1$),Alg.~\ref{alg:zolosign} ($n=4$),Alg.~\ref{alg:zolosign} ($n=8$), Pad\'e ($n=1$), Pad\'e ($n=4$), Pad\'e ($n=8$), Polar (Newton), Polar (Zolo-pd), Polar (SVD), Direct,
Alg.~\ref{alg:zolosign} ($n=1$),Alg.~\ref{alg:zolosign} ($n=4$),Alg.~\ref{alg:zolosign} ($n=8$), Pad\'e ($n=1$), Pad\'e ($n=4$), Pad\'e ($n=8$), Polar (Newton), Polar (Zolo-pd), Polar (SVD), Direct,
Alg.~\ref{alg:zolosign} ($n=1$),Alg.~\ref{alg:zolosign} ($n=4$),Alg.~\ref{alg:zolosign} ($n=8$), Pad\'e ($n=1$), Pad\'e ($n=4$), Pad\'e ($n=8$), Polar (Newton), Polar (Zolo-pd), Polar (SVD), Direct,
Alg.~\ref{alg:zolosign} ($n=1$),Alg.~\ref{alg:zolosign} ($n=4$),Alg.~\ref{alg:zolosign} ($n=8$), Pad\'e ($n=1$), Pad\'e ($n=4$), Pad\'e ($n=8$), Polar (Newton), Polar (Zolo-pd), Polar (SVD), Direct}},
columns/leftcol/.style={string type,column type/.add={}{},column name={$\frac{\pi}{2}-\Theta(A)$}},
columns/leftcol2/.style={string type,column type/.add={}{},column name={Algorithm}},
columns/0/.style={sci,sci e,sci zerofill,precision=1,column type/.add={}{},column name={$\|A-\widehat{V}\widehat{\Lambda}\widehat{V}^*\|$}}, 
columns/1/.style={sci,sci e,sci zerofill,precision=1,column type/.add={}{},column name={$\|\widehat{V}^*\widehat{V}-I\|$}}
]
{eig12.dat}
\hspace{0.1in}
\begin{filecontents*}{eig34.dat}
   4.8396691e-15   4.1678260e-15
   5.2404131e-15   3.5886871e-15
   4.9265919e-15   3.3133413e-15
   5.7029936e-15   4.6506747e-15
   2.3110935e-14   3.6739540e-15
   5.0468173e-14   3.2905510e-15
   4.8118498e-15   4.0932427e-15
   5.5146903e-01   3.3117071e-15
   4.7546264e-01   4.0810150e-15
   2.1461021e-14   1.6832032e-14
   4.5552324e-15   3.2890187e-15
   4.7394360e-15   3.7793946e-15
   4.9034034e-15   3.5299225e-15
   4.7752065e-15   4.0472642e-15
   4.9620358e-15   3.5570197e-15
   5.7615795e-15   3.4962878e-15
   9.2454198e-15   3.3612352e-15
   5.5333108e-15   3.6266558e-15
   6.1784086e-15   3.5299172e-15
   1.3269612e-14   8.6829595e-15
\end{filecontents*}
\pgfplotstabletypeset[
every head row/.style={
after row=\midrule},
every nth row={10}{before row=\midrule},
columns={leftcol2,0,1},
create on use/leftcol2/.style={create col/set list={
Alg.~\ref{alg:zolosign} ($n=1$),Alg.~\ref{alg:zolosign} ($n=4$),Alg.~\ref{alg:zolosign} ($n=8$), Pad\'e ($n=1$), Pad\'e ($n=4$), Pad\'e ($n=8$), Polar (Newton), Polar (Zolo-pd), Polar (SVD), Direct,
Alg.~\ref{alg:zolosign} ($n=1$),Alg.~\ref{alg:zolosign} ($n=4$),Alg.~\ref{alg:zolosign} ($n=8$), Pad\'e ($n=1$), Pad\'e ($n=4$), Pad\'e ($n=8$), Polar (Newton), Polar (Zolo-pd), Polar (SVD), Direct,
Alg.~\ref{alg:zolosign} ($n=1$),Alg.~\ref{alg:zolosign} ($n=4$),Alg.~\ref{alg:zolosign} ($n=8$), Pad\'e ($n=1$), Pad\'e ($n=4$), Pad\'e ($n=8$), Polar (Newton), Polar (Zolo-pd), Polar (SVD), Direct,
Alg.~\ref{alg:zolosign} ($n=1$),Alg.~\ref{alg:zolosign} ($n=4$),Alg.~\ref{alg:zolosign} ($n=8$), Pad\'e ($n=1$), Pad\'e ($n=4$), Pad\'e ($n=8$), Polar (Newton), Polar (Zolo-pd), Polar (SVD), Direct}},
columns/leftcol/.style={string type,column type/.add={}{},column name={$\frac{\pi}{2}-\Theta(A)$}},
columns/leftcol2/.style={string type,column type/.add={}{},column name={Algorithm}},
columns/0/.style={sci,sci e,sci zerofill,precision=1,column type/.add={}{},column name={$\|A-\widehat{V}\widehat{\Lambda}\widehat{V}^*\|$}}, 
columns/1/.style={sci,sci e,sci zerofill,precision=1,column type/.add={}{},column name={$\|\widehat{V}^*\widehat{V}-I\|$}}
]
{eig34.dat}
\caption{Performance of algorithms for computing the unitary eigendecomposition of the matrices~\ref{mat1}-\ref{mat2} (left) and~(\ref{mat3}-\ref{mat4}) (right).  With the exception of the entries labeled ``Direct'', the entries reported in column 1 refer to the algorithms for the unitary sign decomposition used in line~\ref{line:sign} of Algorithm~\ref{alg:eig}.}
\label{tab:eig}
\end{table}

\section{Conclusion} \label{sec:conclusion}

This paper constructed structure-preserving iterations for computing the unitary sign decomposition using rational minimax approximants of the scalar function $\sign(z)$ on the unit circle.  Relative to other structure-preserving iterations, they converge significantly faster, and relative to non-structure-preserving iterations, they exhibit much better numerical stability.  We used our iterations to construct a spectral divide-and-conquer algorithm for the unitary eigendecomposition.

\bibliography{references}

\end{document}